\documentclass{article}
\usepackage{amsmath,amssymb,amsthm,graphicx,float}
\usepackage{subcaption} 
\newtheorem{theorem}{Theorem}
\newtheorem{lemma}[theorem]{Lemma}

\newtheorem{corollary}[theorem]{Corollary}
\newtheorem{conjecture}[theorem]{Conjecture}

\theoremstyle{definition}
\newtheorem{definition}[theorem]{Definition}
\newtheorem{remark}[theorem]{Remark}
\newtheorem{problem}{Problem}
\DeclareMathOperator{\spec}{Spec}

\newcommand{\ztwo}{\mathbb{Z}/2\mathbb{Z}}

\title{Improved bounds on the size of the smallest representation of relation algebra $32_{65}$}
\author{Jeremy F.~Alm, Michael Levet, Saeed Moazami, Jorge Montero-Vallejo,\\ Linda Pham, Dave Sexton, Xiaonan Xu \\  corresponding author email: \texttt{alm.academic@gmail.com}}
\date{}

\begin{document}

\maketitle

\begin{abstract}
In this paper, we shed new light on the spectrum of the relation algebra we call $A_{n}$, which is obtained by splitting the non-flexible diversity atom of $6_{7}$ into $n$ symmetric atoms. Precisely, we show that the minimum value in $\text{Spec}(A_{n})$ is at most $2n^{6 + o(1)}$, which is the first polynomial bound and improves upon the previous bound due to Dodd \& Hirsch (\textit{J. Relational Methods in Computer Science} 2013). We also improve the lower bound to $2n^{2} + 4n + 1$, which is asymptotically double the trivial bound of $n^{2} + 2n + 3$. 

In the process, we obtain stronger results regarding $\text{Spec}(A_{2}) =\text{Spec}(32_{65})$. Namely, we show that $1024$ is in the spectrum, and no number smaller than 26 is in the spectrum. Our improved lower bounds were obtained by employing a SAT solver, which suggests that such tools may be more generally useful in obtaining representation results.
\end{abstract}

\section{Introduction}
Relation algebra $32_{65}$ has atoms $1'$, $a$, $b$, and $c$, all symmetric, with all diversity cycles not involving $a$ forbidden. The atom $a$ is flexible, and $32_{65}$ has the mandatory cycles required to make $a$ flexible and no others. The numbering system for finite integral relation algebras is due to Maddux \cite{Madd}. 

Relation algebra $32_{65}$ was shown in \cite{AMM} to be representable over a finite set, namely a set of 416,714,805,914 points. This was reduced in \cite{DH} to 63,432,274,896 points, which was later reduced to 8192 by the first and sixth authors (unpublished), and finally to 3432 in \cite{MR3951643}.  Here, we give the smallest known representation, over 1024 points. 

There are few published lower bounds in the literature. Most  can be found in \cite{MR1334290}, where the spectrum of every relation algebra with three or fewer atoms is determined. Going up to four atoms increases the difficulty considerably. 

No lower bound on the size of representations of $32_{65}$ has been published. We give a non-trivial such bound for an infinite class of algebras in Section \ref{sec:LB}. 

The importance of the relation algebras studied here, namely $32_{65}$ and those derived from splitting a non-flexible diversity atom of $32_{65}$, stems from the flexible atom conjecture:

\begin{conjecture}[Flexible Atom Conjecture]
Every finite integral relation algebra with a flexible atom has a representation over a finite set. 
\end{conjecture}

In \cite{MadJipTuz}, the finite symmetric integral relation algebras in which every diversity atom is flexible were shown to be finitely representable. In particular, the algebra with $n$ flexible atoms is representable over a set of size $(2+o(1))n^2$. This implies that $32_{65}$ is finitely representable. (A corollary implies that $31_{37}$ is finitely representable as well.) In the present paper, we look at the other extreme, where only one atom is flexible, and the only mandatory cycles are those that involve the flexible atom, i.e., only those that are required to make the atom flexible. We show that in this case as well, the number of points required grows only polynomially in the number of atoms. (See Theorem \ref{ThmPolyBound}.) 

Now we give the requisite definitions. 
\begin{definition}
A relation algebra  is an algebra $\langle A, +, -, ;, \breve{}, 1'\rangle $ such that 
\begin{itemize}
    \item $\langle A, +, -\rangle$ is a Boolean algebra
    \item $\langle A, ;, \breve{}, 1'\rangle $ is an involuted monoid
    \item $x \mathbin{;} (y+z) = x\mathbin{;} y + x\mathbin{;} z$
    \item $(x+y)\;\breve{} = \breve{x} + \breve{y}$
    \item For all $x$, $y$, and $z$, we have $$x\mathbin{;}y \cdot \breve{z} =0 \iff y\mathbin{;}z \cdot \breve{x} =0.$$
    
\end{itemize}
\end{definition}

\begin{definition}
For a relation algebra $A$, any $a \in A$ is an atom if $a \neq 0$ and $b < a$ implies $b=0$; it is called a diversity atom if, in addition, $a\cdot 1' = 0$. 
\end{definition}

\begin{definition}
For diversity atoms $a$, $b$, $c$, the triple $(a,b,c)$, usually written $abc$, is called a diversity cycle. A cycle $abc$ is called forbidden if $a\mathbin{;}b \cdot c = 0$ and mandatory if $a\mathbin{;}b \geq c$. (Note that these are the only possibilities, otherwise $c$ is not an atom.)
\end{definition}

\begin{definition}
We say that a symmetric diversity atom $f$ is \textit{flexible} if for all diversity atoms $a$, $b$, we have that $abf$ is mandatory. 
\end{definition}

\begin{definition}
A relation algebra $A$ is called representable if there is a set $U$ and an equivalence relation $E\subseteq U\times U$ such that $A$ embeds in $$\langle \mathrm{Powerset}(E), \cup,\ ^c, \circ, ^{-1}, Id_E\rangle. $$ 
\end{definition}

In this paper, we will be concerned only with simple RAs, so $E$ can be equal to $U\times U$ for some set $U$. A representation where $E = U\times U$  is called \emph{square}. 

\begin{definition}
Let $A$ be a finite relation algebra.  Then 
\[
    \spec(A)=\{\alpha \leq \omega : A\text{ has a square representation over a set of cardinality }\alpha \}.
\]
\end{definition}

Let $A_n$ denote the integral symmetric relation algebra with atoms $1'$, $r$, $b_1$, \ldots, $b_n$, where a diversity cycle is mandatory if and only if it involves the atom $r$. (So $A_2$ is $32_{65}$.) Let 
\[
    f(n) = \min( \spec(A_n)). 
\]
It was shown in \cite{AMM} that $f(n)$ is finite for all $n$.

Because representing finite integral relation algebras amounts to edge-coloring complete graphs with the diversity atoms, we will use the language of graph theory. So that we can use colors to make pretty pictures, we will refer to $a$ as \emph{red}, to $b$ as \emph{light blue}, and to $c$ as \emph{dark blue}.

\section{An upper bound on $f(n)$}

In this section, we give a representation of $32_{65}$ over 1024 points, and then generalize to give representations of $A_n$ for all $n$. 

Consider $G=(\mathbb{Z}/2\mathbb{Z})^{10}$, and consider the elements as bitstrings.  Define 
\begin{align*}
    R&=\{x \in G : x \text{ has between one and six 1s} \}, \text{ and } \\
    B&=\{x \in G : x \text{ has at least seven 1s} \}. 
\end{align*}

This defines a group representation of $6_7$, which is a super-algebra of $32_{65}$. There exists a way of splitting $B$ into $B_1$ and $B_2$ so that:
\begin{itemize}
    \item $R+B_i=G\smallsetminus\{0\}$, $i=1,2$;
    \item $B_i+B_i = R \cup \{0\}$, $i=1,2$;
    \item $B_1+B_2 = R$.
\end{itemize}

This yields a group representation of $32_{65}$ over $2^{10}=1024$ points, improving the previous smallest-known representation over ${14 \choose 7}=3432$ points \cite{MR3951643}. We note that while the representation given here is smaller, the representation over $3432$ points in \cite{MR3951643} has a nice, compact description. 

The split was found in the following way. The first author checked several million random splits. None of them worked, but some got ``close''. He took one of the close ones and tinkered with it for about three hours until it worked. The curious can view the process in the Jupyter notebook \texttt{32\_65 splitting.ipynb} at \texttt{https://github.com/algorithmachine/RA-32-of-65}. The following Python 3 code can be used to verify that the given split yields a representation. (Bitstrings are encoded as integers between 0 and 1023. Note that the Python operator $\wedge$ denotes the bitwise exclusive-or operation, which is the group operation in our setting.)
\begin{verbatim}
    
def s(X,Y): return {x^y for x in X for y in Y} 
G = set(range(1024)); id = {0}; di = G-id 
b = {127, 223, 239, 251, 253, 255, 367, 
375, 381, 382, 431, 443, 446, 471, 475, 
477, 478, 487, 491, 494, 499, 505, 509, 
607, 635, 637, 639, 701, 702, 703, 719, 
727, 733, 734, 743, 750, 751, 758, 763, 
766, 815, 823, 827, 829, 847, 859, 862, 
863, 877, 879, 883, 886, 887, 890, 893, 
894, 919, 923, 925, 927, 935, 941, 943, 
949, 950, 953, 954, 958, 979, 981, 982, 
990, 991, 995, 1001, 1002, 1003, 1005, 
1011, 1012, 1013, 1014, 1015, 1016, 
1017, 1019, 1021, 1022} 
a = s(b,b)-id; c = di-a-b 
print ( s(a,a)==G, s(a,b)==s(a,c)==di,
s(b,b)==s(c,c)==a|id, s(b,c)==a )

\end{verbatim}

We generalize this argument as follows:

\begin{theorem}\label{thm:big}

For all $n\geq 2$, $A_n$ is representable over $(\ztwo)^{3k+1}$ for sufficiently large $k\in \omega$.  In particular, for $n\geq 14$, it suffices to take $k=n$.

\end{theorem}

\begin{remark}
Theorem \ref{thm:big} tells us that $f(n)$ is at most exponential in $n$. In contrast, the most you could say from \cite{AMM} was that $f(n)$ was bounded above by (roughly) ${15n^2 \choose n}$. See Figure \ref{fig:plot}.
\end{remark}

\begin{proof}
We actually prove something more general: we split both the flexible atom and the non-flexible atom into $n$ parts, since we get this stronger result with essentially no more work. 

We have already shown that for $k = 3$, $A_{2}$ can be realized over $(\mathbb{Z}/2\mathbb{Z})^{3k+1}$. We now argue that for $n \geq 3$, $A_{n}$ can be realized over $(\ztwo)^{3k+1}$ for sufficiently large $k \in \omega$. Our approach is to use the probabilistic method to show that, given a large enough representation of the relation algebra $6_{7}$ over $(\ztwo)^{3k+1}$, the atom $b$ can be partitioned into $n$ parts, as $A_{n}$ is obtained from $6_{7}$ by splitting, as in \cite{MR1052567}. 

Consider $G = (\ztwo)^{3k+1}$. For  $x\in G$, let $x(i)$ denote the $i^\text{th}$ coordinate of $x$.  Let $|x|$ denote $|\{ i:x(i)=1\} |$. Denote $\text{supp}(x) = \{ i : x(i) = 1\}$ to be the support of $x$.  The key idea is the following partition of $G\smallsetminus\{ 0\}$ into two sets $R$ and $B$.

Let:
\begin{align*}
&R=\{ x\in G:1\leq|x|\leq 2k\} \text{ and } \\
&B=\{ x\in G: 2k+1\leq|x|\leq 3k+1\}.
\end{align*}

Then $R+R=G$, $R+B=G\smallsetminus\{ 0\}$, and $B+B=R\cup\{0\}$. As we will see below, $B$ is a sum-free set with high additive energy.

We now split both the ``red'' and ``blue'' atoms of $6_7$ into $n$ atoms and find a representation over a finite set. Namely, we split $R$ and $B$ into $n$ parts $R_1,\ldots,R_n$ and $B_1,\ldots,B_n$ uniformly at random. We need to count the ``witnesses" to the ``needs" of each element.  We will show that each need is witnessed at least $2^k$ times. Consider the following cases.

\begin{itemize}
\item \textbf{Case 1:} We count witnesses for $R\subseteq B+B$. Let $z \in R$, and denote $\ell := |z|$. We consider two sub-cases: whether $1 \leq \ell \leq k$, and whether $k+1 \leq \ell \leq 2k$.

\begin{itemize}
    \item \textbf{Case 1.1:} Suppose first that $1\leq \ell\leq k$. We construct $x,y$ randomly so that $z=x+y$. For each $i\in \text{supp}(z)$, we choose uniformly at random whether:
\begin{align*}
    x(i)=1 &\text{ and }y(i)=0\\
    &\text{OR}\\
    x(i)=0 &\text{ and }y(i)=1.
\end{align*}

As $|z| = \ell$, this yields $2^\ell$ possible selections. For the $k-\ell$ left-most indices $j\notin \text{supp}(z)$, we choose uniformly at random whether:
\begin{align*}
    x(j) &=1=y(j)\\
    &\text{OR}\\
    x(j) &=0 =y(j).
\end{align*}

As there are $k-\ell$ positions, there are $2^{k-\ell}$ possible selections. For the remaining $2k+1$ positions $i$, we let $x(i)=1=y(i)$. Thus, we have that $x, y \in B$. By the rule of product, we obtain $2^{\ell} \cdot 2^{k-\ell} = 2^{k}$ possible selections. It follows that there are at least $2^{k}$ witnesses for $z$.

\item \textbf{Case 1.2:} Suppose now that $k+1\leq \ell\leq 2k$. For the $k$ \emph{least} indices $i\in \text{supp}(z)$, we choose uniformly at random whether:
\begin{align*}
   x(i)=1 &\text{ and } y(i)=0\\
    &\text{OR}\\
    x(i)=0 &\text{ and } y(i)=1
\end{align*}

For the remaining $\ell-k$ indices $i\in \text{supp}(z)$, let $x(i)=1$ and $y(i)=0$, or $x(i)=0$ and $y(i)=1$ in such a way that ensures that both $x$ and $y$ have at least $\ell-k$ 1's in coordinates in $\text{supp}(z)$. Then for all indices $j\notin \text{supp}(z)$, let $x(j)=1=y(j)$. There were $k$ flips, so there are at least $2^k$ witnesses.
\end{itemize}

\noindent \\ It follows that if $z \in R$, there are at least $2^{k}$ ways to witness $z$ as the sum $x+y$, where $x, y \in B$. \\

\item \textbf{Case 2:} Now let us consider witnesses to $B\subseteq B+R$. Let $z \in B$, and denote $\ell := |z|$. By the definition of $B$, we have that $2k+1 \leq \ell \leq 3k.$ We randomly construct $x\in B$, $y\in R$ so that $z=x+y$.

For the $2k+1$ indices $i\in \text{supp}(z)$ of \emph{least} index, set $x(i)=1$ and $y(i)=0$. For the remaining $\ell-(2k+1)$ indices $i\in \text{supp}(z)$, we choose uniformly at random whether:
\begin{align*}
    x(i)=1 &\text{ and }y(i)=0\\
    &\text{OR}\\
    x(i)=0 &\text{ and } y(i)=1.
\end{align*}

For each index $j\notin \text{supp}(z)$, we choose uniformly at random whether:
\begin{align*}
x(j) &=1=y(j)\\
    &\text{OR}\\
    x(j) &=0 =y(j).
\end{align*}

Again, there were $k$ flips, so we have at least $2^k$ witnesses.

\item \textbf{Case 3:} Next, let us consider witnesses to $B\subseteq R+R$. Let $z\in B$. We construct $x,y\in R$ so that $z = x+y$. For every $j\not \in \text{supp}(z)$, set $x(j)=0=y(j)$. This is to ensure that $x, y \in R$. For the smallest $k$ indices $i\in \text{supp}(z)$, we choose uniformly at random whether:
\begin{align*}
    x(i)=1 &\text{ and }y(i)=0\\
    &\text{OR}\\
    x(i)=0 &\text{ and } y(i)=1.
\end{align*}

For the remaining indices $i\in \text{supp}(z)$, we choose uniformly at random whether:
\begin{align*}
    x(i)=1 &\text{ and }y(i)=0\\
    &\text{OR}\\
    x(i)=0 &\text{ and } y(i)=1
\end{align*}
in such a way that ensures that neither $x$ nor $y$ receives more than $2k$ 1's. Clearly, there are at least $2^k$ witnesses.

\item \textbf{Case 4:} Now we consider witnesses for $R\subseteq B+R$. 

Let $z\in R$, and denote $\ell := |z|$. We build $x\in B$, $y\in R$ so that $z=x+y$. We consider the following sub-cases, namely whether $1 \leq \ell \leq k$, and whether $k+1 \leq \ell \leq 2k$.
\begin{itemize}
    \item \textbf{Case 4.1:} First, consider the case where $1\leq \ell\leq k$. For $i \in \text{supp}(z)$, set $x(i)=1$ and $y(i)=0$. For $j\notin \text{supp}(z)$, choose $2k+1-\ell$ of the $3k+1-\ell$ indices, and set $x(j)=1=y(j)$. Set all others to $x(j)=0=y(j)$. Since ${3k+1-\ell\choose 2k+1-\ell} \geq {2k\choose k}>2^k$, we have at least $2^k$ witnesses. 
    
    \item \textbf{Case 4.2:} Now consider the case where $k+1\leq n\leq 2k$. For the smallest $k$ indices $i\in \text{supp}(z)$, set $x(i)=1$ and $y(i)=0$. For the remaining $\ell-k$ indices $i\in \text{supp}(z)$, we choose uniformly at random whether:
\begin{align*}
    x(i)=1 &\text{ and }y(j)=0\\
    &\text{OR}\\
    x(i)=0 &\text{ and } y(j)=1.
\end{align*}

Now we choose $k+1$ of the remaining $3k+1-\ell$ indices $j\notin \text{supp}(z)$. There are ${3k+1-\ell\choose k+1}$ choices, which ranges between ${k+1\choose k+1}$ and ${2k+1\choose k+1}$. Therefore there are at least $2^{\ell-k}\cdot {3k+1-\ell\choose k+1}$ witnesses. It is not hard to check that for $0\leq N\leq k$, ${k+1+N\choose k+1}>2^N$, and therefore $2^{\ell-k}\cdot{3k+1-\ell\choose k+1}>2^{\ell-k}\cdot 2^{2k-\ell}=2^k$.

\end{itemize}

\item \textbf{Case 5:} Finally, we consider witnesses for $R\subseteq R+R$. Let $z\in R$, and denote $\ell := |z|$. We construct $x,y\in R$ so that $z=x+y$. We consider the following cases: whether $1 \leq \ell \leq k$, and whether $k+1 \leq \ell \leq 2k$.
\begin{itemize}
    \item \textbf{Case 5.1:} First, consider the case where $1\leq \ell\leq k$. For each $i\in \text{supp}(z)$, set $x(i)=1$ and $y(i)=0$. Then for the smallest $2k$ indices outside of $\text{supp}(z)$, choose $k$ of then. For each such selected $j$, set $x(j)=1=y(j)$, and $x(j)=0=y(j)$ otherwise. This gives at least ${2k\choose k}>2^k$ witnesses.
    
    \item \textbf{Case 5.2:} We next consider the case where $k+1\leq \ell\leq 2k$. For each $i \in \text{supp}(z)$, we choose uniformly at random whether:
\begin{align*}
    x(i)=1 &\text{ and }y(i)=0\\
    &\text{OR}\\
    x(i)=0 &\text{ and } y(i)=1.
\end{align*}
This gives $2^\ell>2^k$ witnesses.
\end{itemize}
\end{itemize}

\noindent \\ Now we are ready to compute the probability that our random partition $R_1,\ldots,R_n$ and $B_1,\ldots ,B_n$ \emph{fails} to be a representation. Let $z\in (\ztwo)^{3k+1} \smallsetminus \{\textbf{0}\}$. If $z\in R_i$, then $z$ has $3n^2$ ``needs":
\begin{itemize}
    \item $\forall i,j \ z\in R_i+R_j$
    \item $\forall i,j \ z\in R_i+B_j$
    \item $\forall i,j \ z\in B_i+B_j$.
\end{itemize}
If $z\in B_j$, then $z$ has $2n^2$ ``needs":
\begin{itemize}
    \item $\forall i,j \ z\in R_i+R_j$
    \item $\forall i,j \ z\in R_i+B_j$.
\end{itemize}
So $3n^2$ is a bound on the number of ``needs". Fix $z$, and let $x, y \in G$ such that $x + y = z$. The probability that the edge $xy$ witnesses a fixed need is $1/n^{2}$. So the probability that the edge $xy$ does \textit{not} witness a fixed need is $(1 - 1/n^{2})$. For a particular need, there are at least $2^{k}$ witnesses. As we color the edges uniformly at random, the probability that a fixed need of $z$ is unsatisfied is at most:
\[
\left( 1 - \frac{1}{n^{2}}\right)^{2^{k}}.
\]

As $z$ has at most $3n^{2}$ needs, we have that
\begin{align*}
Pr[z\text{ has an unsatisfied need}] \leq 3n^2\left(1-\frac{1}{n^2}\right)^{2^k}.
\end{align*}

Thus
\begin{align}
Pr[\exists z\text{ with an unsatisfied need}] &\leq \sum_z 3n^2\left(1-\frac{1}{n^2}\right)^{2^k}\\
    &=2^{3k+1}\cdot 3n^2\left(1-\frac{1}{n^2}\right)^{2^k} \label{thmbig2} \\
    &\leq2^{3(k+1)}\cdot n^2\left(1-\frac{1}{n^2}\right)^{2^k}.\label{star}
\end{align}

We want \eqref{star} to be less than 1, which is equivalent to its logarithm being less than zero:
\begin{align*}
    \log \left(2^{3(k+1)}\cdot n^2\cdot \left(1-\frac{1}{n^2}\right)^{2^k}\right)<0 &\Longleftrightarrow 3(k+1)\log 2+2\log n+2^k\log \left(\frac{n^2-1}{n^2}\right)<0\\
    &\Longleftrightarrow 3(k+1)\log 2+2\log n<2^k\log\left(\frac{n^2}{n^2-1}\right)
\end{align*}
Now assuming $3 \leq n \leq k$, we have that
\begin{align*}
3(k+1)\log 2+2\log n &<3(k+1)+k\\
&\leq 5k. 
\end{align*}
Thus
\begin{align*}
    2^k\cdot \log \left(\frac{n^2}{n^2-1}\right) &=2^k[\log (n^2)-\log (n^2-1)]\\
    &>2^k\cdot\frac{1}{n^2},
\end{align*}
where the last inequality is due to the fact that $\log(t+1)-\log(t) < 1/t$, which follow the concavity $\log$. So we need $5k<\frac{2^k}{n^2}$. Setting $k=n$, we have $5n^3<2^n$, which holds for all $n\geq 14$. Hence taking $k=n$ gives a non-zero probability that a random partition yields a representation.

\end{proof}

\noindent The construction in the proof of Theorem \ref{thm:big} provides the bound $f(n) \leq 2^{3n+1}$ for $n \geq 14$. By fine-tuning our choice of $k$, we  obtain polynomial bounds on $f(n)$. Let $k = (2+o(1))\log(n)$. We have by (\ref{thmbig2}) that
\begin{align*}
Pr[\exists z\text{ with an unsatisfied need}] &\leq 2^{3k+1}\cdot 3n^2\left(1-\frac{1}{n^2}\right)^{2^k} \\
&= 2^{3(2+o(1))\log(n) + 1} \cdot 3n^{2} \cdot \left(1-\frac{1}{n^2}\right)^{2^{(2+o(1))\log(n)}} \\
&= 2n^{6 + o(1)} \cdot 3n^{2} \cdot \left(1-\frac{1}{n^2}\right)^{2^{(2+o(1))\log(n)}} \\
&= 6n^{8+o(1)} \cdot \left(1-\frac{1}{n^2}\right)^{2^{(2+o(1))\log(n)}}. 
\end{align*}

Now we have that
\begin{align} \label{whplimit}
\lim_{n \to \infty} 6n^{8+o(1)} \cdot \left(1-\frac{1}{n^2}\right)^{2^{(2+o(1))\log(n)}} = 0.
\end{align}

So choosing $k = (2+o(1))\log(n)$ yields the following. 
\begin{theorem} \label{ThmPolyBound}
For $n$ sufficiently large, we have that $f(n) \leq 2n^{6 + o(1)}$.
\end{theorem}

We note that the threshold $n_{0}$ for which Theorem \ref{ThmPolyBound} applies is quite large. For instance, choosing $k = 3\log(n)$ yields that $f(n) \leq 2n^{9}$, which holds for all $n \geq 91$. So in choosing $k = (2+o(1))\log(n)$, the bound of $f(n) \leq 2n^{6+o(1)}$ holds for all $n \geq n_{0} \geq 91$, where $n_{0}$ depends on $o(1)$. This contrasts with the bound in Theorem \ref{thm:big}, which holds for all $n \geq 14$. Furthermore, calibrating our choice of $k = c \log(n)$ failed to yield improvements on $f(3) \leq 2^{16}$ and $f(4) \leq 2^{19}.$

It is natural to ask whether modifying our choice of $k$ in this construction will yield additional improvements in the upper bound for $f(n)$. If we take $k = 2\log(n)$ rather than $k = (2+o(1))\log(n)$, we have that
\begin{align*}
\lim_{n \to \infty} 6n^{8} \cdot \left(1-\frac{1}{n^2}\right)^{2^{2\log(n)}} = \infty.
\end{align*}

The key reason behind this is that
\begin{align*}
\lim_{n \to \infty} \left(1-\frac{1}{n^2}\right)^{2^{2\log(n)}} = \lim_{n \to \infty} \left(1-\frac{1}{n^2}\right)^{n^{2}} = \frac{1}{e}.
\end{align*}

This suggests that further analyzing the Boolean cube is unlikely to yield additional improvements on the upper bound for $f(n)$.

We also note that (\ref{whplimit}) yields that as $n \to \infty$,  the probability that there exists $z \in (\ztwo)^{3k+1}$ with an unsatisfied need goes to $0$. So with high probability, a random partition yields a representation of $A_n$. We record this observation with the following corollary.

\begin{corollary}
Suppose that we split both the ``red'' and ``blue'' atoms of $6_7$ into $n$ atoms, as in the proof of Theorem \ref{thm:big}. Namely, we split $R$ and $B$ into $n$ parts $R_1,\ldots,R_n$ and $B_1,\ldots,B_n$ uniformly at random. With high probability, we have that such a random split is a representation of a relation algebra containing a subalgebra isomorphic to $A_n$.
\end{corollary}

\begin{figure}
    \centering
    \includegraphics[width=5in]{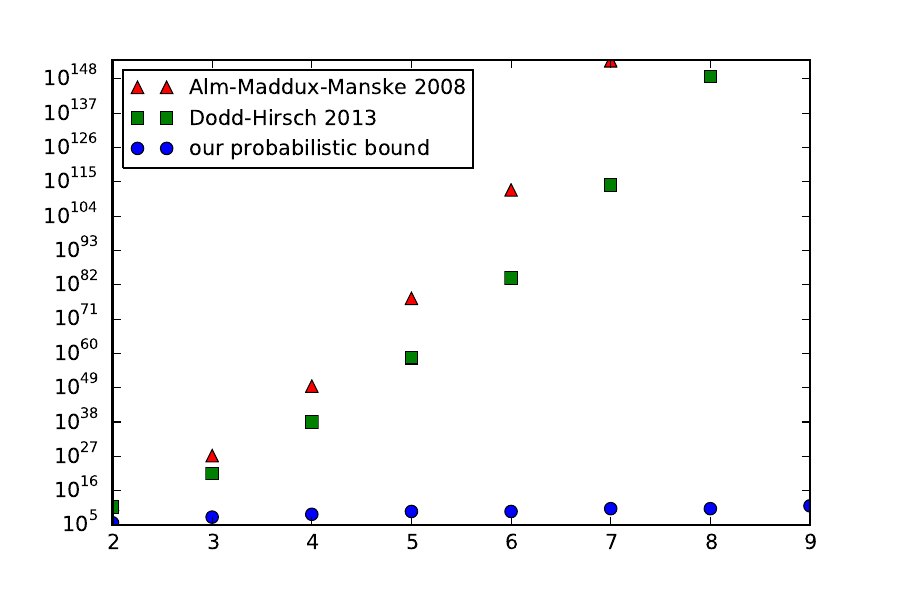}
    \caption{Upper bound on $f(n)$ vs.~$n$}
    \label{fig:plot}
\end{figure}

\section{A lower bound}\label{sec:LB}

In this section, we consider representations of $32_{65}$ as edge-colorings of $K_n$ with all mandatory triangles present and no all-blue triangles. Note that blue triangles are forbidden even if they contain edges of differing shades of blue. In other words, every triangle must contain a red edge.  

We now make our representation precise. Let $\rho:32_{65} \to \text{Powerset}(U\times U)$, where $U = \{x_0, \dots , x_{n-1}\}$, be a representation. Then label the vertices of $K_n$ with $\{x_0, \dots , x_{n-1}\}$, and let the color of edge $x_ix_j$ be the atom $z$ such that $(x_i, x_j)  \in \rho(z)$. 

\begin{lemma}
$\spec(32_{65}) \subseteq \{11, \ldots\} \cup \{\omega\}$.
\end{lemma}
\begin{proof}
There must be some red edge $x_0x_1$. Any red edge has nine ``needs''. There must be nine points that witness these needs, which together with $x_0$ and $x_1$ make a total of 11 points. (See Figure \ref{fig:needs}.)
\end{proof}

We can easily obtain a slight improvement using the classical Ramsey number $R(m,n)$.
\begin{lemma}
$11 \not\in \spec(32_{65}) $.
\end{lemma}

\begin{proof}
We know that at least 11 points are required. Since $R(4,3)=9$, and there are no all-blue triangles, there must be a red $K_4$. Let $x_0x_1$ be an edge in this red $K_4$. Then $x_0x_1$ must have its red-red need met twice, hence there must be ten points besides $x_0$ and $x_1$.
\end{proof}

\begin{lemma}
In any representation of $32_{65}$, for every red edge there is a red $K_4$ that is vertex-disjoint from it. In particular, off of every red edge $x_0 x_1$ one can find the configuration depicted in Figure \ref{fig:needs-red}.
\end{lemma}

\begin{proof}
Let $x_0x_1$ be red, with witnesses to all needs as in Figure \ref{fig:needs}. Then $\{x_2,x_3,x_4,x_5\}$ induce a red $K_4$, since any blue edge among them would create an all-blue triangle with $x_0$ (and also with $x_1$). Furthermore, any edge running from any of $x_2,x_3,x_4,x_5$ to any of $x_6,x_7,x_8,x_9$ must be red, since any such blue edge  would create an all-blue triangle with either $x_0$ (for $x_7$ and $x_9$) or $x_1$ (for $x_6$ and $x_8$).  Thus we have the configuration  depicted in Figure \ref{fig:needs-red}.
\end{proof}

\begin{lemma}\label{lem:17}
$12,13,14,15,16 \not\in \spec(32_{65})$.
\end{lemma}
\begin{proof}
Consider the configuration depicted in Figure \ref{fig:needs-red}. The edge $x_2x_5$ is red. Then $x_0$ and $x_1$ both witness the light-blue-dark-blue need, while $x_3,x_4,x_6,x_7,x_8$ and $x_9$ all witness the red-red need.  There are seven needs yet unsatisfied. The remaining vertex $x_{10}$ could witness some need, but vertices $x_{11}$ through $x_{16}$ will have to be added. Thus there are at least 17 points. See Figure \ref{fig:17}. 
\end{proof}

Lemma \ref{lem:17} generalizes nicely as follows.

\begin{theorem}\label{thm:main}
For all $n$, $f(n) \geq 2n^2+4n+1$.
\end{theorem}

Note that the trivial bound is $n^2+2n+3$, roughly half the bound in Theorem \ref{thm:main}. 

\begin{proof}
Call the shades of blue $b_1$ through $b_n$.
Fix a red edge $x_0x_1$. Let $BB$ denote the set of vertices that witness a blue-blue need for $x_0x_1$, and let $RB$ denote the set of vertices that witness either a red-blue need or a blue-red need for $x_0x_1$. $BB$ induces a red clique, and all edges from $BB$ to $RB$ are red. Note that $|BB|=n^2$ and $RB=2n$. This gives the trivial lower bound of $n^2+2n+3$.

Let $u\in BB$ witness $b_1$-$b_1$ for $x_0x_1$ and let $v\in BB$ witness $b_2$-$b_2$ for $x_0x_1$. The edge $uv$ is red, hence has $(n+1)^2$ needs. Both $x_0$ and $x_1$ witness the same $b_1$-$b_2$ need, and all points in $BB$ and $RB$ (besides $u$ and $v$) witness the red-red need. Hence there must be at least $(n+1)^2 - 2$ points outside of $\{x_0,x_1\}\cup BB \cup RB$. Hence there are at least $2 + n^2 + 2n + (n+1)^2 - 2 = 2n^2 + 4n + 1$ points.
\end{proof}

\begin{remark}
Note that if two points $u, v$ satisfy a red-blue need for $x_{0}x_{1}$, then $uv$ is necessarily red. Otherwise, $uvx_{1}$ would form a blue triangle There are $n$ points that satisfy the red-blue need for $x_{0}x_{1}$. As the points in $BB$ form a red clique of size $n^{2}$, we obtain the following. 
\end{remark}

\begin{corollary}
In any representation of $A_n$, the clique number of the red subgraph of the underlying graph of the representation is at least $n^2 + n$. 
\end{corollary}

\section{SAT solver results}
In this section, we improve the lower bound on $f(2)$ using a SAT solver.

\begin{lemma}
$17 \not\in\spec(32_{65})$.
\end{lemma}

\begin{proof}
We build an unsatisfiable boolean formula $\Phi$ whose satisfiability is a necessary condition for $32_{65}$ to be representable over 17 points. For all $0\leq i < j < 17$ and $k=0,1,2$, define a boolean $\phi_{i,j,k}$. We interpret $\phi_{i,j,0}$ being TRUE to mean that $x_ix_j$ is red, $\phi_{i,j,1}$ being TRUE to mean that $x_ix_j$ is light blue, and $\phi_{i,j,2}$ being TRUE to mean that $x_ix_j$ is dark blue. Then define 
\[
    \Phi_0 = \bigwedge_{i<j} [ (\phi_{i,j,0} \vee \phi_{i,j,1} \vee \phi_{i,j,2})  \wedge (\neg\phi_{i,j,0} \vee \neg\phi_{i,j,1})  \wedge (\neg\phi_{i,j,0} \vee \neg\phi_{i,j,2})  \wedge (\neg\phi_{i,j,1} \vee \neg\phi_{i,j,2}) ]
\]
Then $\Phi_0$ asserts that for each $i<j$, exactly one of $\phi_{i,j,0}$, $\phi_{i,j,1}$, and $\phi_{i,j,2}$ is TRUE. 

Consider the subgraph depicted in Figure \ref{fig:needs-red}. Let
\begin{itemize}
    \item $R = \{(i,j): i<j,\ x_ix_j \text{ is red}\}$
    \item $Bl = \{(i,j): i<j,\ x_ix_j \text{ is light blue}\}$
    \item $Bd = \{(i,j): i<j,\ x_ix_j \text{ is dark blue}\}$
\end{itemize}

Define 
\[
    \Phi_1 = \left(\bigwedge_{(i,j)\in R} \phi_{i,j,0}\right) \wedge \left(\bigwedge_{(i,j)\in Bl} \phi_{i,j,1}\right) \wedge \left(\bigwedge_{(i,j)\in Bd} \phi_{i,j,2} \right)
\]
Then $\Phi_1$ asserts that any  edges colored in Figure \ref{fig:needs-red} are colored correctly.

Finally, define 

\[
    \Phi_2 = \bigwedge_{(i,j)\in R}\left[ \bigwedge_{c_1,c_2\in \{0,1,2\}} \left(\bigvee_{i\neq k \neq j} \phi_{i,k,c_1}\wedge\phi_{k,j,c_2}\right)\right]
\]
Then $\Phi_2$ asserts that every red edge in Figure \ref{fig:needs-red} has its needs satisfied.

Let $\Phi = \Phi_0 \wedge \Phi_1 \wedge \Phi_2$.

$\Phi$ has been verified by SAT solver to be unsatisfiable when there are 17 points.

\end{proof}

\begin{corollary}
In any representation of $32_{65}$, the clique number of the red subgraph of the underlying graph of the representation is at least six. 
\end{corollary}

\begin{proof}
From the previous lemma, we see that at least 18 points are required to represent $32_{65}$. Since $R(6,3)=18$, there must be a red $K_6$.
\end{proof}

Thus we can always find the subgraph depicted in Figure \ref{fig:K6}.

Unfortunately, $\Phi$ is satisfiable on 18 or more points.  We must add more clauses to make $\Phi$ unsatisfiable.  First, expand $R$ to include the red $K_6$ as in Figure \ref{fig:K6}.  Second, we add clauses to forbid all-blue triangles:

\[
    \Phi_3 = \bigwedge_{i<j<k} ( \phi_{i,j,0} \vee \phi_{i,k,0} \vee \phi_{j,k,0} )
\]

\begin{lemma}
Let $\Phi = \Phi_0 \wedge \Phi_1 \wedge \Phi_2\wedge \Phi_3$. Then for $n=18$, 19, $\Phi$ is unsatisfiable. Hence 18, 19, $\not\in\spec(32_{65})$.
\end{lemma}

\begin{proof}
We have verified the unsatisfiability of $\Phi$ via SAT solver.
\end{proof}


The SAT solver runs into computational issues on $n \geq 20$ points, so we add more clauses to limit the search space. $\Phi_4$ and $\Phi_5$ assert that every light blue edge and every dark blue edge in Figure \ref{fig:needs-red} has its needs satisfied, respectively, and $\Phi_6$ asserts that every edge that is not pre-colored has its needs satisfied:
Define $BN = \{(0,0), (0,1), (0,2), (1,0), (2,0)\}$. 

\[
    \Phi_4 = \bigwedge_{(i,j)\in Bl}\left[ \bigwedge_{(c_1,c_2)\in BN} \left(\bigvee_{i\neq k \neq j} \phi_{i,k,c_1}\wedge\phi_{k,j,c_2}\right)\right]
\]

\[
    \Phi_5 = \bigwedge_{(i,j)\in Bd}\left[ \bigwedge_{(c_1,c_2)\in BN} \left(\bigvee_{i\neq k \neq j} \phi_{i,k,c_1}\wedge\phi_{k,j,c_2}\right)\right]
\]

\begin{align*}
     \Phi_6 = \bigwedge_{(i,j)\notin R\cup Bl \cup Bd } & \phi_{i,j,0} \wedge \left[ \bigwedge_{(c_1,c_2)\in \{0,1,2\}} \left(\bigvee_{i\neq k \neq j} \phi_{i,k,c_1}\wedge\phi_{k,j,c_2}\right)\right] \\
     \vee \ &\phi_{i,j,1} \wedge \left[ \bigwedge_{(c_1,c_2)\in BN} \left(\bigvee_{i\neq k \neq j} \phi_{i,k,c_1}\wedge\phi_{k,j,c_2}\right)\right] \\
      \vee \  &\phi_{i,j,2} \wedge \left[ \bigwedge_{(c_1,c_2)\in BN} \left(\bigvee_{i\neq k \neq j} \phi_{i,k,c_1}\wedge\phi_{k,j,c_2}\right)\right] \\
\end{align*}

\begin{lemma}
Let $\Phi = \Phi_0 \wedge \Phi_1 \wedge \Phi_2\wedge \Phi_3\wedge \Phi_4 \wedge \Phi_5\wedge \Phi_6$. Then for $n\leq 25$, $\Phi$ is unsatisfiable. Hence  $\{20,21,22,23,24,25\}\not\subseteq\spec(32_{65})$.
\end{lemma}

\begin{proof}
We have verified the unsatisfiability of $\Phi$ via SAT solver.
\end{proof}

\begin{remark}
We note that the Ramsey number $R(7, 3) = 23$. As $23 \not \in \spec(32_{65})$, we obtain the following.
\end{remark}

\begin{corollary} 
In any representation of $32_{65}$, the clique number of the red subgraph of the underlying graph of the representation is at least seven. 
\end{corollary}

\section{Summary and open problems}

We summarize our work as follows.

\begin{theorem}
We have $26 \leq f(2) \leq 1024$, and: 
\begin{enumerate}
\item $f(n) \geq 2n^2+4n+1$ for all $n$.
\item $f(n) \leq 2n^{6 + o(1)}$.
\end{enumerate} 
\end{theorem}

\begin{problem}
Is $f(2) < 1000$?
\end{problem}

\begin{problem}
Is $32_{65}$ representable over $(\ztwo)^m$ for $m<10$? The natural thing to try -- using the construction from the proof of Theorem \ref{thm:big}, with $k=2$ (hence $m=7$) -- doesn't work; we checked all partitions. But there may some other representation. 
\end{problem}

\begin{problem}
Can some modification of the technique used in \cite{MR3951643} give a smaller representation of $32_{65}$? The most obvious thing to try -- replacing ${[14] \choose 7}$ by ${[11] \choose 6}$ -- doesn't work. 
\end{problem}

\begin{problem}\label{prob:3132}
Which has the smaller minimal representation, $31_{37}$ or $32_{65}$? While $32_{65}$ has atoms $r, b_1, b_2$, all symmetric, with all-blue triangles forbidden,  $31_{37}$ has atoms $r, b, b\,\breve{}$,  with all-blue triangles forbidden. The atom $r$ is flexible in both cases. The lower bound proven in Theorem \ref{thm:main} applies to representations of $31_{37}$ as well. The only (small) finite representation  known to the authors is over $\mathbb{Z}/33791\mathbb{Z}$. 
\end{problem}

\section{Data availability statement}

 The datasets generated during the current study are available in the GitHub repository, \texttt{https://github.com/algorithmachine/RA-32-of-65}.

\begin{figure}
    \centering
    \begin{minipage}{0.45\textwidth}
        \centering
        \includegraphics[width=2.5in]{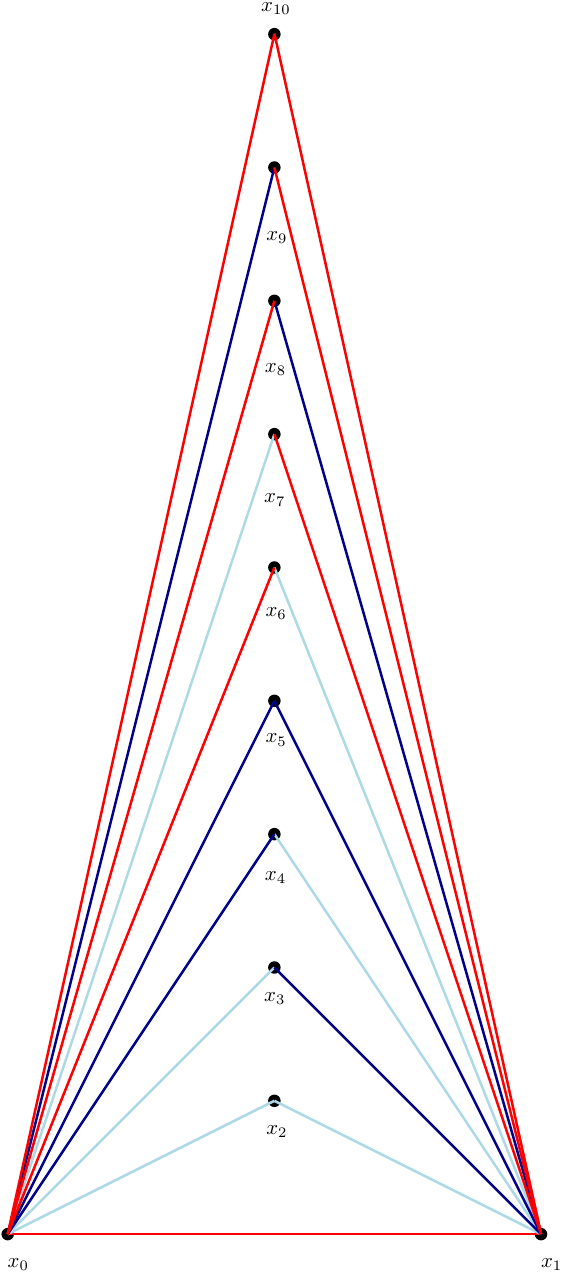}
    \caption{The needs of a red edge}
    \label{fig:needs}
    \end{minipage}\hfill
    \begin{minipage}{0.45\textwidth}
        \centering
        \includegraphics[width=2.5in]{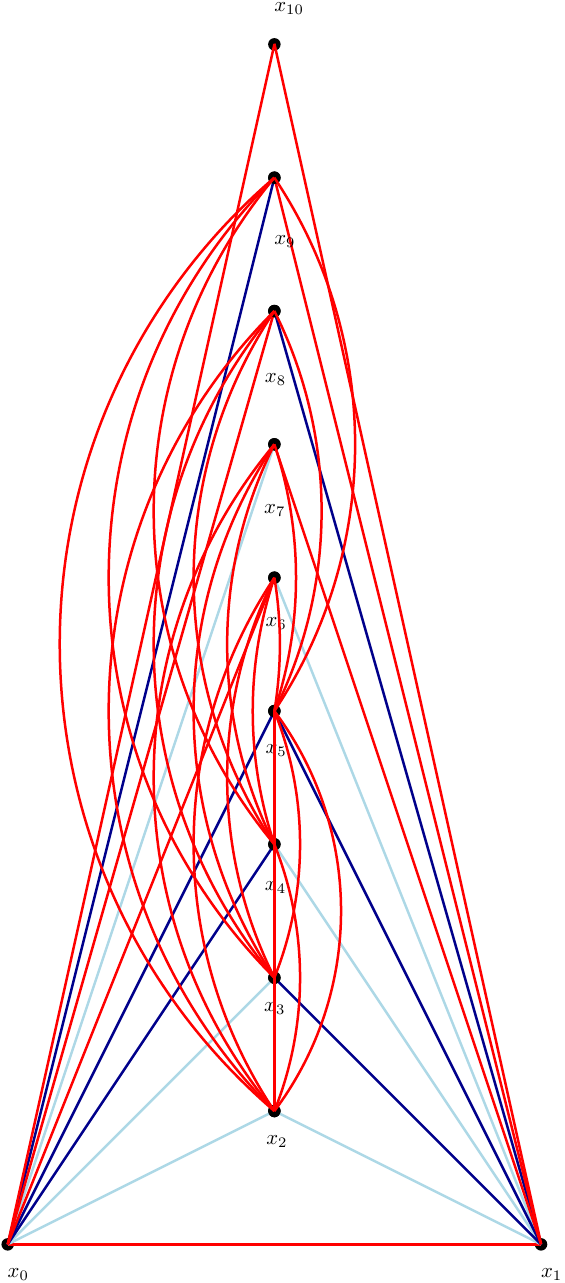}
    \caption{Subgraph which must appear off of any red edge}
    \label{fig:needs-red}
    \end{minipage}
\end{figure}


\begin{figure}
    \centering
    \begin{minipage}{0.45\textwidth}
        \centering
       \includegraphics[width=2.5in]{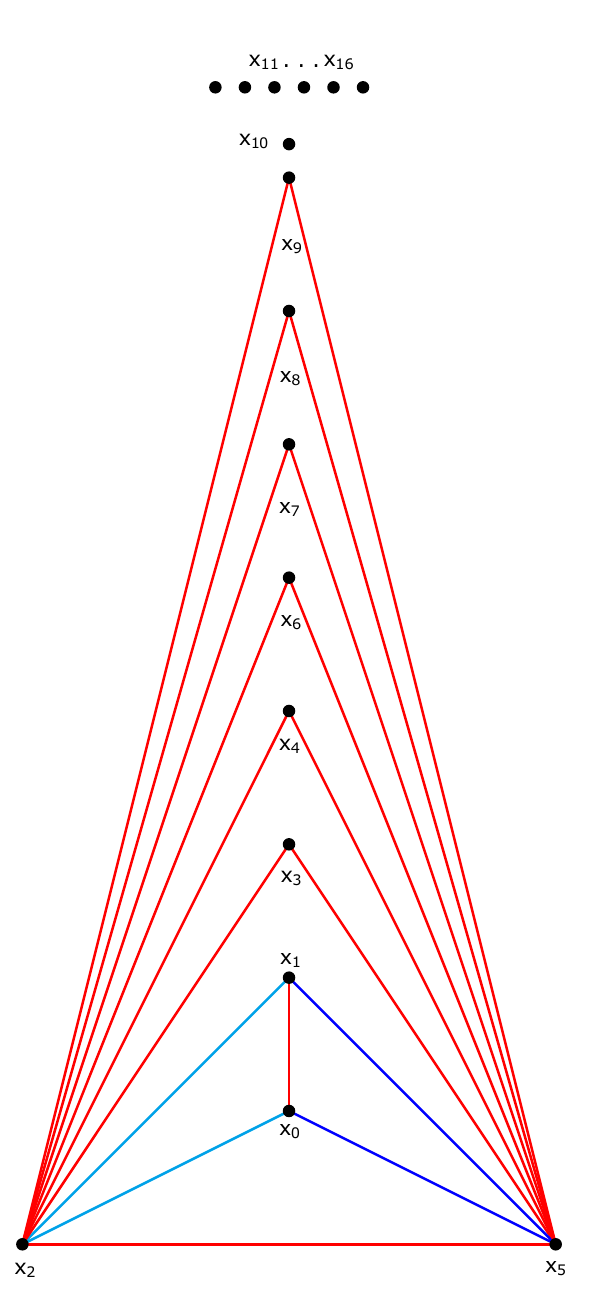}
    \caption{witnesses to the needs of $x_2x_5$}
    \label{fig:17}
    \end{minipage}\hfill
    \begin{minipage}{0.45\textwidth}
        \centering
      
    \includegraphics[width=2.5in]{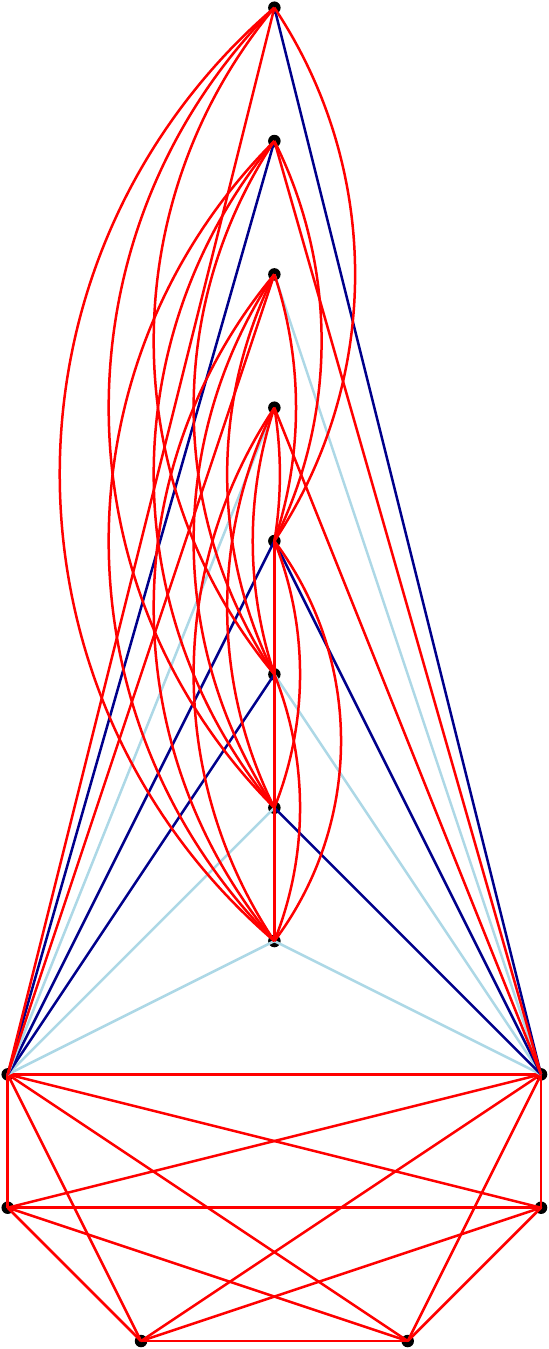}
    \caption{Mandatory subgraph with red $K_6$}
    \label{fig:K6}

    \end{minipage}
\end{figure}
\newpage 
\bibliographystyle{amsplain}
\bibliography{references}

\end{document}